\documentclass[12pt]{amsart}
\usepackage{amssymb}
\usepackage{amsmath}
\usepackage[dvips]{graphicx}
\usepackage{subfigure}
\usepackage{hyperref}
\usepackage[capitalize,nameinlink,noabbrev,nosort]{cleveref}
\usepackage[dvipsnames]{xcolor}
\usepackage{mathtools}
\usepackage{enumitem}
\usepackage[headings]{fullpage}
\usepackage{ytableau}
\usepackage{yhmath}
\usepackage{mathdots}
\usepackage{fullpage}
\usepackage{graphics}
\usepackage{tikz}
\usetikzlibrary{arrows}
\usetikzlibrary{decorations.markings}
\usetikzlibrary{tikzmark,decorations.pathreplacing}
\usetikzlibrary{shapes.geometric}
\numberwithin{equation}{section}
\usepackage{pstricks}
\usepackage{comment} 

\newcommand{\ds}{\displaystyle}

\renewcommand{\sp}{\mathrm{sp}}
\DeclareMathOperator{\sgn}{sgn}
\newtheorem{thm}{Theorem}[section]

\newtheorem{cor}[thm]{Corollary}
\newtheorem{lem}[thm]{Lemma}

\crefname{thm}{Theorem}{Theorems}
\newcommand{\z}{\bar{z}}

\newcommand{\q}{\bar{q}}
\newcommand{\sighhh}{\widehat{\sigma}}
\newcommand{\Pf}{\mathop{\mathrm{Pf}}}
\renewcommand{\t}{\bar{t}}
\renewcommand{\u}{\bar{u}}

\newcommand{\Bulki}{\raisebox{-1.2mm}{\begin{tikzpicture}[scale=0.60]
\draw (-0.5,2) -- (-0.5,2.5);
\draw (-0.5,2) -- (0,2);
\draw (-0.5,2) -- (-1,2);
\draw (-0.5,2) -- (-0.5,1.5);
\filldraw[black] (-0.5,2) circle (2pt);
\draw [-latex](-0.5,2)--(-0.5,2.50);
\draw [-latex](-0.5,1.50)--(-0.5,1.90);
\draw [-latex](-0.5,2)--(-0,2);
\draw [-latex](-1,2)--(-0.60,2);
\end{tikzpicture}}}
\newcommand{\Bulkii}{\raisebox{-1.2mm}{\begin{tikzpicture}[scale=0.60]
\draw (-0.5,2) -- (-0.5,2.5);
\draw (-0.5,2) -- (0,2);
\draw (-0.5,2) -- (-1,2);
\draw (-0.5,2) -- (-0.5,1.5);
\filldraw[black] (-0.5,2) circle (2pt);
\draw [-latex](-0.5,2.5)--(-0.5,2.1);
\draw [-latex](-0.5,2)--(-0.5,1.50);
\draw [-latex](0,2)--(-0.4,2);
\draw [-latex](-0.5,2)--(-1,2);
\end{tikzpicture}}}
\newcommand{\Bulkiii}{\raisebox{-1.2mm}{\begin{tikzpicture}[scale=0.60]
\draw (-0.5,2) -- (-0.5,2.5);
\draw (-0.5,2) -- (0,2);
\draw (-0.5,2) -- (-1,2);
\draw (-0.5,2) -- (-0.5,1.5);
\filldraw[black] (-0.5,2) circle (2pt);
\draw [-latex](-0.5,2)--(-0.5,2.50);
\draw [-latex](-0.5,1.5)--(-0.5,1.9);
\draw [-latex](0,2)--(-0.4,2);
\draw [-latex](-0.5,2)--(-1,2);
\end{tikzpicture}}}
\newcommand{\Bulkiv}{\raisebox{-1.2mm}{\begin{tikzpicture}[scale=0.60]
\draw (-0.5,2) -- (-0.5,2.5);
\draw (-0.5,2) -- (0,2);
\draw (-0.5,2) -- (-1,2);
\draw (-0.5,2) -- (-0.5,1.5);
\filldraw[black] (-0.5,2) circle (2pt);
\draw [-latex](-0.5,2.5)--(-0.5,2.1);
\draw [-latex](-0.5,2)--(-0.5,1.50);
\draw [-latex](-0.5,2)--(-0,2);
\draw [-latex](-1,2)--(-0.60,2);
\end{tikzpicture}}}
\newcommand{\Bulkv}{\raisebox{-1.2mm}{\begin{tikzpicture}[scale=0.60]
\draw (-0.5,2) -- (-0.5,2.5);
\draw (-0.5,2) -- (0,2);
\draw (-0.5,2) -- (-1,2);
\draw (-0.5,2) -- (-0.5,1.5);
\filldraw[black] (-0.5,2) circle (2pt);
\draw [-latex](-0.5,2)--(-0.5,2.50);
\draw [-latex](-0.5,2)--(-0.5,1.50);
\draw [-latex](0,2)--(-0.40,2);
\draw [-latex](-1,2)--(-0.60,2);
\end{tikzpicture}}}
\newcommand{\Bulkvi}{\raisebox{-1.2mm}{\begin{tikzpicture}[scale=0.60]
\draw (-0.5,2) -- (-0.5,2.5);
\draw (-0.5,2) -- (0,2);
\draw (-0.5,2) -- (-1,2);
\draw (-0.5,2) -- (-0.5,1.5);
\filldraw[black] (-0.5,2) circle (2pt);
\draw [-latex](-0.5,2.5)--(-0.5,2.1);
\draw [-latex](-0.5,1.5)--(-0.5,1.9);
\draw [-latex](-0.5,2)--(-0,2);
\draw [-latex](-0.5,2)--(-1,2);
\end{tikzpicture}}}

\newcommand{\Lftout}{\raisebox{0.1mm}{\begin{tikzpicture}[scale=0.60]
\draw (-0.5,2) -- (-0.5,2.5);
\draw (-0.5,2) -- (0,2);
\filldraw[black] (-0.5,2) circle (2pt);
\draw [-latex](-0.5,2)--(-0.5,2.50);
\draw [-latex](-0.5,2)--(0,2);
\end{tikzpicture}}}

\newcommand{\Lftin}{\raisebox{0.1mm}{\begin{tikzpicture}[scale=0.60]
\draw (-0.5,2) -- (-0.5,2.5);
\draw (-0.5,2) -- (0,2);
\filldraw[black] (-0.5,2) circle (2pt);
\draw [-latex](-0.5,2.5)--(-0.5,2.1);
\draw [-latex](0,2)--(-0.4,2);
\end{tikzpicture}}}

\newcommand{\Lftup}{\raisebox{0.1mm}{\begin{tikzpicture}[scale=0.60]
\draw (-0.5,2) -- (-0.5,2.5);
\draw (-0.5,2) -- (0,2);
\filldraw[black] (-0.5,2) circle (2pt);
\draw [-latex](-0.5,2)--(-0.5,2.50);
\draw [-latex](0,2)--(-0.4,2);
\end{tikzpicture}}}

\newcommand{\Lftdown}{\raisebox{0.1mm}{\begin{tikzpicture}[scale=0.60]
\draw (-0.5,2) -- (-0.5,2.5);
\draw (-0.5,2) -- (0,2);
\filldraw[black] (-0.5,2) circle (2pt);
\draw [-latex](-0.5,2.5)--(-0.5,2.1);
\draw [-latex](-0.5,2)--(-0,2);
\end{tikzpicture}}}

\newcommand{\Rt}{\raisebox{0.5mm}{
\begin{tikzpicture}[scale=0.60]
    \draw (-0.5,2) -- (-1,2);
    \filldraw[black] (-0.5,2) circle (2pt);
    \draw [-latex](-0.5,2)--(-1,2);
\end{tikzpicture}
}}
\newcommand{\K}{\raisebox{0.25mm}{
\begin{tikzpicture}[scale=0.60]
    \draw (-0.5,2) -- (-0.5,2.5);
    \filldraw[black] (-0.5,2.5) circle (2pt);
    \draw [-latex](-0.5,2)--(-0.5,2.45);
\end{tikzpicture}
}}

\title{Off-diagonally symmetric alternating sign matrices}

\author{Nishu Kumari}
\address{Nishu Kumari, Department of Mathematics, 
University of Vienna, Austria.}
\email{nishukumari@alum.iisc.ac.in}

\date{\today}
\begin{document}
\begin{abstract}
A diagonally symmetric alternating sign matrix (DSASM) is a symmetric matrix with entries $-1$, $0$ and $1$, where the nonzero entries alternate in sign along each row and column, and the sum of the entries in each row and column equals $1$. An off-diagonally symmetric alternating sign matrix (OSASM) is a DSASM, where the number of nonzero diagonal entries is 0 for even-order matrices and 1 for odd-order matrices. Kuperberg (Ann. Math., 2002) studied even-order OSASMs and and derived a product formula for counting the number of OSASMs of any fixed even order. In this work, we provide a product formula for the number of odd-order OSASMs of any fixed order. Additionally, we present an algebraic proof of a symmetry property for even-order OSASMs. This resolves all the three conjectures of Behrend, Fischer, and Koutschan (arXiv, 2023) regarding the exact enumeration of OSASMs.
\end{abstract}
\subjclass[2020]{05E10, 05E05, 05B20, 15B35}
\keywords{Alternating sign matrices, product formula}
\maketitle
\section{Introduction}
\label{sec:intro}
An {alternating sign matrix} (ASM) is a square matrix that meets the following criteria:
\begin{itemize}
    \item each entry is either $1$, $0$ or $-1$,
    \item the nonzero entries alternate in sign along each row and column,
    \item the sum of the entries in each row and column is 1.
\end{itemize}

We use ASM$(n)$ to represent the set of ASMs of order $n$, for any integer $n \geq 1$. Mills, Robbins and Rumsey~\cite{mills1982proof, mills1983alternating} introduced the concept of ASMs and conjectured that the number of ASMs of order $n$ is given by 
\[
|\text{ASM}(n)|=\prod_{i=0}^{n-1} \frac{(3i+1)!}{(n+i)!}.
\]

The conjecture was first resolved by Zeilberger~\cite{zeilberger1994proof}, with subsequent proofs provided by Kuperberg~\cite{kuperberg1996another}, Fischer~\cite{fischer2006number, fischer2007new, fischer2016short}, and Fischer and Konvalinka~\cite{fischer2019bijective, fischer2020mysterious, fischer2022bijective}. 
For a history of ASMs, see~\cite{bressoud1999proofs}.

The group of symmetries of a square, known as the dihedral group $D_4$ of order 8, acts naturally on the set of ASMs. Exact formulas for the number of ASMs invariant under the action of various subgroups of $D_4$
  have been derived in a series of papers~\cite{andrews1994plane, kuperberg2002symmetry, robbins2000symmetry, okada2006enumeration, behrend2017diagonally, mills1986self, razumov2006enumeration, razumov2006enumerations}. Since the number of invariant ASMs is the same for any two conjugate subgroups, this results in eight distinct symmetry classes of ASMs. These classes are as follows: unrestricted ASMs, vertically (or horizontally) symmetric ASMs, vertically and horizontally symmetric ASMs, half-turn symmetric ASMs, quarter-turn symmetric ASMs, diagonally (or antidiagonally) symmetric ASMs, diagonally and antidiagonally symmetric ASMs, and totally symmetric ASMs. For a review of the general approach to obtaining enumeration formulas for various symmetry classes of ASMs, see~\cite[Section 2]{behrend2023diagonally}.

We focus on off-diagonally symmetric ASMs, a subclass within the symmetry class of diagonally symmetric ASMs. A {diagonally symmetric ASM} (DSASM) is an ASM that remains invariant under matrix transposition. Recently, Behrend, Fischer, and Koutschan~\cite{behrend2023diagonally} established a Pfaffian formula for the number of DSASMs of any fixed order.   
An {off-diagonally symmetric ASM} (OSASM), as defined by Kuperberg~\cite{kuperberg2002symmetry}, is a DSASM in which all diagonal entries are zero. We use OSASM$(n)$ to represent the set of OSASMs of order $n$. 
Kuperberg~\cite[Theorem 2]{kuperberg2002symmetry} demonstrated that the number of OSASMs of order $2n$ is given by:
\[
|\text{OSASM}(2n)|=\prod_{i=1}^{n} \frac{(6i-2)!}{(2n+2i)!}.
\]
Observe that it is not possible for each diagonal entry of an odd-order DSASM to be zero. If all diagonal entries of a DSASM of order $n$ are zero, then the sum of all entries equals twice the sum of all strictly upper triangular entries. Since the sum of all entries is equal to $n$, it follows that $n$ must be even. Consequently, define an {odd-order OSASM} as an odd-order DSASM in which exactly one diagonal entry is nonzero. Odd-order OSASMs, to our knowledge, are first considered by Behrend, Fischer and Koutschan~\cite{behrend2023diagonally}. They also conjecture~\cite[Conjecture 15]{behrend2023diagonally} that the number of $(2n+1) \times (2n+1)$ OSASMs is given by
\begin{equation}
\label{OSams-odd}
|\text{OSASM}(2n+1)|=\frac{2^{n-1} (3n+2)!}{(2n+1)!}\prod_{i=1}^{n} \frac{(6i-2)!}{(2n+2i+1)!}.
\end{equation}
We give an algebraic proof of \eqref{OSams-odd} in \cref{sec:product}. To this end, we define the DSASM and OSASM generating functions in \cref{sec:gf} and the DSASM and OSASM partition functions in \cref{sec:pf}. In \cref{sec:product}, we first prove another conjecture~\cite[Conjecture 17]{behrend2023diagonally} (see \cref{odd-osasm}) and then use it to write a proof of \eqref{OSams-odd}. 
Lastly, we prove a symmetry property for even-order OSASMs conjectured by Behrend, Fischer and Koutschan~\cite[Conjecture 16]{behrend2023diagonally} in \cref{sec:symm}. Combinatorial proofs for these conjectures are deferred to a forthcoming paper, written jointly with Fischer~\cite{ilse-2025}.

\section{The DSASM and OSASM generating functions}
\label{sec:gf}
In this section, we review the definitions of the DSASM and OSASM generating functions considered in ~\cite{behrend2023diagonally}.

Throughout the paper, for an indeterminate $x$, a ring $R$ and an expression $F \in R[x,x^{-1}]$, $F\big|_{x=a}$ is the algebraic expression obtained by substituting
$x=a$ in $F$. It is being used in \eqref{XOdef} for the first time.

As noted in \cref{sec:intro}, a DSASM
is an ASM that remains invariant under matrix transposition. We denote the set of DSASMs of order $n$ by DSASM$(n)$. For examples, DSASM$(1)$, DSASM$(2)$ and DSASM$(3)$ are
\[
\left\{
\left(
\begin{array}{c}
1
\end{array}
\right)
\right\},
\qquad 
\left\{
\left(
\begin{array}{cc}
    1 & 0 \\
    0 & 1
\end{array}
\right),
\left(
\begin{array}{cc}
    0 & 1 \\
    1 & 0
\end{array}
\right)
\right\}
\]
and
\[
\left
\{
\left(
\begin{array}{ccc}
1&0&0\\
0&1&0\\
0&0&1
\end{array} \right),
\left( 
\begin{array}{ccc}
1&0&0\\
0&0&1\\
0&1&0
\end{array} 
\right),
\left( \begin{array}{ccc}
0&1&0\\
1&0&0\\
0&0&1
\end{array} \right),
\left( \begin{array}{ccc}
0&0&1\\
0&1&0\\
1&0&0
\end{array} 
\right),
\left( 
\begin{array}{ccc}
0&1&0\\
1&-1&1\\0&1&0
\end{array}
\right)\right\}
\]
respectively.

It is straightforward to observe that each of the first and last rows and columns of any ASM contains exactly one 1, with all other entries being 0.
Consider the following three statistics for a DSASM 
$A=(A_{i,j})_{1 \leq i,j \leq n}$, labeled 
$R(A)$, $S(A)$ and $T(A)$:
\begin{equation}
\label{defnRST}
\begin{split}
    R(A)=&\text{ the number of nonzero entries in the strictly upper triangular part of } A\\
    = & \, \sum_{1 \leq i<j \leq n} |A_{i,j}|,\\
   S(A) = & \text{ the number of nonzero entries on the diagonal of } A\\
    = & \,\sum_{i=1}^n |A_{i,i}|,\\
    T(A)=&\text{ the column index of the 1 in the first row of }A.
    \end{split}
\end{equation}
The {DSASM generating function} $X_n(r,s,t)$ is defined as 
\begin{equation}
\label{gf}
X_n(r,s,t)=\sum_{A \, \in \, \text{DSASM}(n)} r^{R(A)} s^{S(A)}
t^{T(A)}, \end{equation}
where $r,s$ and $t$ are indeterminates. For examples, 
\begin{align*}
    X_1(r,s,t)=\,st,
    X_2(r,s,t)=\,s^2t+r t^2,
    X_3(r,s,t)=\,s^3t+rst+rst^2+rst^3+r^2st^2.
\end{align*}

We refer the reader to~\cite[Section 4]{behrend2023diagonally} for Pfaffian expressions of $X_n(r,s,t)$. 

As mentioned in \cref{sec:intro}, an OSASM of even order is an even-order DSASM where all diagonal entries are zero, while an OSASM of odd order is an odd-order DSASM with exactly one nonzero diagonal entry. We denote the set of OSASMs of order $n$ by OSASM$(n)$.
For examples,  OSASM$(1)$, OSASM$(2)$ and OSASM$(3)$ are
\[
\left\{
\left(
\begin{array}{c}
1
\end{array}
\right)
\right\},
\qquad 
\left\{
\left(
\begin{array}{cc}
    0 & 1 \\
    1 & 0
\end{array}
\right)
\right\}
\]
and
\[
\left
\{
\left( \begin{array}{ccc}
1&0&0\\0&0&1\\0&1&0\end{array} \right),\left( \begin{array}{ccc}
0&1&0\\1&0&0\\0&0&1\end{array} \right),
\left( \begin{array}{ccc}
0&0&1\\0&1&0\\1&0&0\end{array} 
\right),\left( \begin{array}{ccc}
0&1&0\\1&-1&1\\0&1&0\end{array}
\right)\right\}
\]
respectively. Define the {OSASM generating function} $X_n^{\mathrm{O}}(r,t)$ as 
\[
X_n^{\mathrm{O}}(r,t)=\sum_{A \, \in \, \text{OSASM}(n)}r^{R(A)} 
t^{T(A)}, 
\]
for indeterminates $r$ and $t$. For examples,
\begin{align*}
    X_1^{\mathrm{O}}(r,t)=\,t,
    X_2^{\mathrm{O}}(r,t)=\,r t^2,
    X_3^{\mathrm{O}}(r,t)=\,rt+rt^2+rt^3+r^2t^2.
\end{align*}
Note that
\begin{equation}
\label{XOdef}
X_n^{\mathrm{O}}(r,t)= \begin{cases}
   X_n(r,0,t) & n \text{ is even},\\
  \left(X_n(r,s,t)/s\right)\big|_{s=0} & n \text{ is odd}.
\end{cases}
\end{equation}
Now we recall the following product formula derived by Behrend, Fischer and Koutschan~\cite[Equation (128)]{behrend2023diagonally}  for $X_{2n}^{\mathrm{O}}(1,-1)$ using a generating function identity of $X_{2n}^{\mathrm{O}}(1,t)$ given by Razumov and Stroganov~\cite{razumov2004refined}. 
\begin{equation}
\label{XO}
X_{2n}^{\mathrm{O}}(1,-1)= \frac{(3n-1)!}{2^n (2n-1)!}
\prod_{i=1}^{n-1} \frac{(6i-2)!}{(2n+2i-1)!}.
\end{equation}
\section{The DSASM and OSASM partition functions}
\label{sec:pf}
In this section, we revisit the definitions of the DSASM and OSASM partition functions as presented in~\cite[Section 5]{behrend2023diagonally}. 

Throughout the paper, we use the following notation. Define $\bar{x}=1/x$, $\sigma(x)=x-\bar{x}$, and $\sighhh(x)=\sigma(x)/\sigma(q^4)$, where $x$ is an indeterminate and $q$ is an arbitrary constant. 
For a positive integer $m$, define
\[
\chi_{\mathrm{even}}(m)=\begin{cases}
    1 & m \text{ is even,}\\
    0 & \text{otherwise.}
\end{cases}
\]
For a skew-symmetric matrix $M=\left(M_{i,j}\right)_{1 \leq i,j \leq 2m}$, the {Pfaffian} of $M$ is defined as 
\[
\Pf(M)=\sum_{\left\{\{i_1,j_1\},\dots,\{i_m,j_m\}\right\}}
\sgn(\sigma) M_{i_1,j_1} \cdots M_{i_m,j_m},
\]
where the sum is over all partitions $\left\{\{i_1,j_1\},\dots,\{i_m,j_m\}\right\}$ of $\{1,\dots,2m\}$
into two-element subsets such that $i_k<j_k$, $k=1,\dots,m$ and $i_1<i_2<\dots<i_m$, and $\sigma$ is the permutation that maps $1$ to $i_1$, $2$ to $j_1$, $\dots$, $2m-1$ to $i_m$ and $2m$ to $j_m$. 
For a skew-symmetric matrix $M$, \begin{equation}
\Pf(M)^2=\det(M).
\end{equation}
For a triangular array $B=\left(B_{i,j}\right)_{1 \leq i<j \leq 2b}$, the Pfaffian is the Pfaffian of the skew-symmetric matrix whose strictly upper triangular part is $B$. For any $k_1,\dots,k_{2b}$,
\begin{equation}
\label{out}
\Pf_{1 \leq i<j \leq 2b} \left( k_i k_j B_{i,j}\right)= \prod_{i=1}^{2b} k_i \Pf_{1 \leq i<j \leq 2b} \left(B_{i,j}\right).
\end{equation}

\vspace{0.15in}

Behrend, Fischer and Koutschan~\cite{behrend2023diagonally} introduced a vertex model whose configurations correspond bijectively to DSASMs of any fixed order. 
The configurations of the vertex model are defined on a {grid graph}
$T_n$, which consists of the following vertices: 
top vertices at $(0,j)$, for $j=1,\dots,n$, left boundary vertices
at $(i,i)$, for $i=1,\dots,n$, 
bulk vertices at $(i,j)$, for $1\le i<j\le n$, 
and right boundary vertices 
at $(i,n+1)$, for $i=1,\dots,n$ (see \cref{fig:1}).

\begin{center}
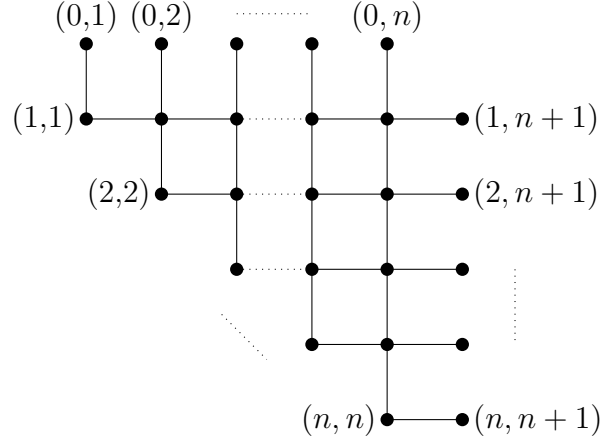
\begin{figure}[htbp]
\begin{tikzpicture}
 \node[above] at (-1,-7) {(0,1)};
 \node[above] at (0,-7) {(0,2)};
  \node[above] at (3,-7) {$(0,n)$};
  \node[left] at (-1,-8) {(1,1)};
  \node[left] at (0,-9) {(2,2)};
  \node[left] at (3,-12) {$(n,n)$};
  \node[right] at (4,-8) {$(1,n+1)$};
  \node[right] at (4,-9) {$(2,n+1)$};
  \node[right] at (4,-12) {$(n,n+1)$};
  \draw  node[fill,circle,inner sep=0pt,minimum size=5pt] at (-1,-7){}; 
  \draw  node[fill,circle,inner sep=0pt,minimum size=5pt] at (0,-7){}; 
  \draw  node[fill,circle,inner sep=0pt,minimum size=5pt] at (1,-7){}; 
  \draw  node[fill,circle,inner sep=0pt,minimum size=5pt] at (2,-7){}; 
  \draw  node[fill,circle,inner sep=0pt,minimum size=5pt] at (3,-7){}; 
   \draw  node[fill,circle,inner sep=0pt,minimum size=5pt] at (-1,-8){};
    \draw  node[fill,circle,inner sep=0pt,minimum size=5pt] at (0,-8){};
     \draw  node[fill,circle,inner sep=0pt,minimum size=5pt] at (1,-8){};
      \draw  node[fill,circle,inner sep=0pt,minimum size=5pt] at (2,-8){};
      \draw  node[fill,circle,inner sep=0pt,minimum size=5pt] at (3,-8){};
       \draw  node[fill,circle,inner sep=0pt,minimum size=5pt] at (4,-8){};
       \draw  node[fill,circle,inner sep=0pt,minimum size=5pt] at (0,-9){};
       \draw  node[fill,circle,inner sep=0pt,minimum size=5pt] at (1,-9){};
       \draw  node[fill,circle,inner sep=0pt,minimum size=5pt] at (2,-9){};
       \draw  node[fill,circle,inner sep=0pt,minimum size=5pt] at (3,-9){};
       \draw  node[fill,circle,inner sep=0pt,minimum size=5pt] at (4,-9){};
        \draw  node[fill,circle,inner sep=0pt,minimum size=5pt] at (1,-10){};
       \draw  node[fill,circle,inner sep=0pt,minimum size=5pt] at (2,-10){};
       \draw  node[fill,circle,inner sep=0pt,minimum size=5pt] at (3,-10){};
       \draw  node[fill,circle,inner sep=0pt,minimum size=5pt] at (4,-10){};
       \draw  node[fill,circle,inner sep=0pt,minimum size=5pt] at (2,-11){};
       \draw  node[fill,circle,inner sep=0pt,minimum size=5pt] at (3,-11){};
       \draw  node[fill,circle,inner sep=0pt,minimum size=5pt] at (4,-11){};
       \draw  node[fill,circle,inner sep=0pt,minimum size=5pt] at (3,-12){};
       \draw  node[fill,circle,inner sep=0pt,minimum size=5pt] at (4,-12){};
\draw[dotted] (1,-6.6)--(2,-6.6);
\draw (-1,-7)--(-1,-8);
 \draw (0,-7)--(0,-8)--(0,-9);
    \draw (1,-7)--(1,-8)--(1,-9)--(1,-10);
    \draw (2,-7)--(2,-8)--(2,-9)--(2,-10)--(2,-11);
    \draw (3,-7)--(3,-8)--(3,-9)--(3,-10)--(3,-11)--(3,-12);
    \draw (-1,-8)--(0,-8)--(1,-8);
    \draw[dotted] (1,-8)--(2,-8);
    \draw (2,-8)--(3,-8)--(4,-8);
    \draw (0,-9)--(1,-9);
    \draw[dotted] (1,-9)--(2,-9);
    \draw (2,-9)--(3,-9)--(4,-9);
    \draw[dotted] (1,-10)--(2,-10);
    \draw (2,-10)--(3,-10)--(4,-10);
    \draw (2,-11)--(3,-11)--(4,-11);
    \draw (3,-12)--(4,-12);
    \draw[dotted] (4.7,-10)--(4.7,-11);
    \draw[dotted] (0.8,-10.6)--(1.4,-11.2);
\end{tikzpicture}
\caption{A grid graph $T_n$}
\label{fig:1}
\end{figure}
 \end{center}
 
A {six-vertex configuration} on $T_n$ is defined as an orientation of the edges of $T_n$ that satisfies the following three conditions:
\begin{enumerate}
\item Each bulk vertex has two incoming edges and two outgoing edges.
\item Each edge incident to a top vertex is directed upward.
\item Each edge incident to a right boundary vertex is directed leftward.
\end{enumerate}

A {local configuration} in a configuration $C$ is the orientations of the edges incident to the vertex. 
Thus, the possible local configurations are as follows:
\begin{itemize}
    \item \K\ at the top vertices, 
    \item \Rt\ at the right boundary vertices,
    \item \Lftup, \Lftdown, \Lftout\ or \Lftin\ at the left boundary vertices,
    \item \Bulkv, \, \Bulkvi, \, \Bulki, \, \Bulkii,\, \Bulkiii\ or  \Bulkiv \ at the bulk vertices. 
\end{itemize}

We denote the set of configurations on $T_n$ by $\mathrm{6V}(n)$. For examples,
\[
\mathrm{6V}(1)=
\left\{
\raisebox{-2mm}{
\begin{tikzpicture}[scale=0.50]
\draw (-1,2) -- (-1,3);
\draw (-1,2)-- (0,2);
\draw [-latex](-1,2)--(-1,2.7);
\filldraw[black] (-1,2) circle (2pt);
\filldraw[black] (-1,3) circle (2pt);
\filldraw[black] (0,2) circle (2pt);
\draw [-latex](0,2)--(-0.7,2);
\end{tikzpicture}
}
\right\},\quad
\mathrm{6V}(2)=\left\{\raisebox{-4.5mm}{
\begin{tikzpicture}[scale=0.50]
\draw (-1,2) -- (-1,3);
\draw (-1,2)-- (0,2)--(1,2);
\draw (0,3)--(0,2)--(0,1);
\draw (0,1)--(1,1);
\draw [-latex](-1,2)--(-1,2.7);
\draw [-latex](0,2)--(0,2.7);
\draw [-latex](0,1)--(0,1.7);
\filldraw[black] (-1,2) circle (2pt);
\filldraw[black] (-1,3) circle (2pt);
\filldraw[black] (0,2) circle (2pt);
\filldraw[black] (0,3) circle (2pt);
\filldraw[black] (1,2) circle (2pt);
\filldraw[black] (0,1) circle (2pt);
\filldraw[black] (1,1) circle (2pt);
\draw [-latex](0,2)--(-0.7,2);
\draw [-latex](1,2)--(0.3,2);
\draw [-latex](1,1)--(0.3,1);
\end{tikzpicture}{\raisebox{4mm}{, }}
\begin{tikzpicture}[scale=0.50]
\draw (-1,2) -- (-1,3);
\draw (-1,2)-- (0,2)--(1,2);
\draw (0,3)--(0,2)--(0,1);
\draw (0,1)--(1,1);
\draw [-latex](-1,2)--(-1,2.7);
\draw [-latex](0,2)--(0,2.7);
\draw [-latex](0,2)--(0,1.3);
\filldraw[black] (-1,2) circle (2pt);
\filldraw[black] (-1,3) circle (2pt);
\filldraw[black] (0,2) circle (2pt);
\filldraw[black] (0,3) circle (2pt);
\filldraw[black] (1,2) circle (2pt);
\filldraw[black] (0,1) circle (2pt);
\filldraw[black] (1,1) circle (2pt);
\draw [-latex](-1,2)--(-0.3,2);
\draw [-latex](1,2)--(0.3,2);
\draw [-latex](1,1)--(0.3,1);
\end{tikzpicture}
}
\right\},
\]
\[\mathrm{6V}(3)=\left\{\raisebox{-6.7mm}{
\begin{tikzpicture}[scale=0.50]
\draw (-1,2) -- (-1,3);
\draw (-1,2)-- (0,2)--(1,2)--(2,2);
\draw (0,3)--(0,2)--(0,1);
\draw (1,3)--(1,2)--(1,1)--(1,0);
\draw (0,1)--(1,1)--(2,1);
\draw (1,0)--(2,0);
\draw [-latex](-1,2)--(-1,2.7);
\draw [-latex](0,2)--(0,2.7);
\draw [-latex](1,2)--(1,2.7);
\draw [-latex](0,1)--(0,1.7);
\draw [-latex](1,1)--(1,1.7);
\draw [-latex](1,0)--(1,0.7);
\filldraw[black] (-1,2) circle (2pt);
\filldraw[black] (-1,3) circle (2pt);
\filldraw[black] (0,2) circle (2pt);
\filldraw[black] (0,3) circle (2pt);
\filldraw[black] (1,2) circle (2pt);
\filldraw[black] (0,1) circle (2pt);
\filldraw[black] (1,1) circle (2pt);
\filldraw[black] (1,3) circle (2pt);
\filldraw[black] (2,2) circle (2pt);
\filldraw[black] (2,1) circle (2pt);
\filldraw[black] (2,0) circle (2pt);
\filldraw[black] (1,0) circle (2pt);
\draw [-latex](2,2)--(1.3,2);
\draw [-latex](2,1)--(1.3,1);
\draw [-latex](2,0)--(1.3,0);
\draw [-latex](0,2)--(-0.7,2);
\draw [-latex](1,2)--(0.3,2);
\draw [-latex](1,1)--(0.3,1);
\end{tikzpicture}{\raisebox{9mm}{, }}
\begin{tikzpicture}[scale=0.50]
\draw (-1,2) -- (-1,3);
\draw (-1,2)-- (0,2)--(1,2)--(2,2);
\draw (0,3)--(0,2)--(0,1);
\draw (1,3)--(1,2)--(1,1)--(1,0);
\draw (0,1)--(1,1)--(2,1);
\draw (1,0)--(2,0);
\draw [-latex](-1,2)--(-1,2.7);
\draw [-latex](0,2)--(0,2.7);
\draw [-latex](1,2)--(1,2.7);
\draw [-latex](0,1)--(0,1.7);
\draw [-latex](1,1)--(1,1.7);
\draw [-latex](1,1)--(1,0.3);
\filldraw[black] (-1,2) circle (2pt);
\filldraw[black] (-1,3) circle (2pt);
\filldraw[black] (0,2) circle (2pt);
\filldraw[black] (0,3) circle (2pt);
\filldraw[black] (1,2) circle (2pt);
\filldraw[black] (0,1) circle (2pt);
\filldraw[black] (1,1) circle (2pt);
\filldraw[black] (1,3) circle (2pt);
\filldraw[black] (2,2) circle (2pt);
\filldraw[black] (2,1) circle (2pt);
\filldraw[black] (2,0) circle (2pt);
\filldraw[black] (1,0) circle (2pt);
\draw [-latex](2,2)--(1.3,2);
\draw [-latex](2,1)--(1.3,1);
\draw [-latex](2,0)--(1.3,0);
\draw [-latex](0,2)--(-0.7,2);
\draw [-latex](1,2)--(0.3,2);
\draw [-latex](0,1)--(0.7,1);
\end{tikzpicture}{\raisebox{9mm}{, }}
\begin{tikzpicture}[scale=0.50]
\draw (-1,2) -- (-1,3);
\draw (-1,2)-- (0,2)--(1,2)--(2,2);
\draw (0,3)--(0,2)--(0,1);
\draw (1,3)--(1,2)--(1,1)--(1,0);
\draw (0,1)--(1,1)--(2,1);
\draw (1,0)--(2,0);
\draw [-latex](-1,2)--(-1,2.7);
\draw [-latex](0,2)--(0,2.7);
\draw [-latex](1,2)--(1,2.7);
\draw [-latex](0,2)--(0,1.3);
\draw [-latex](1,1)--(1,1.7);
\draw [-latex](1,0)--(1,0.7);
\filldraw[black] (-1,2) circle (2pt);
\filldraw[black] (-1,3) circle (2pt);
\filldraw[black] (0,2) circle (2pt);
\filldraw[black] (0,3) circle (2pt);
\filldraw[black] (1,2) circle (2pt);
\filldraw[black] (0,1) circle (2pt);
\filldraw[black] (1,1) circle (2pt);
\filldraw[black] (1,3) circle (2pt);
\filldraw[black] (2,2) circle (2pt);
\filldraw[black] (2,1) circle (2pt);
\filldraw[black] (2,0) circle (2pt);
\filldraw[black] (1,0) circle (2pt);
\draw [-latex](2,2)--(1.3,2);
\draw [-latex](2,1)--(1.3,1);
\draw [-latex](2,0)--(1.3,0);
\draw [-latex](-1,2)--(-0.3,2);
\draw [-latex](1,2)--(0.3,2);
\draw [-latex](1,1)--(0.3,1);
\end{tikzpicture}{\raisebox{9mm}{, }}
\begin{tikzpicture}[scale=0.50]
\draw (-1,2) -- (-1,3);
\draw (-1,2)-- (0,2)--(1,2)--(2,2);
\draw (0,3)--(0,2)--(0,1);
\draw (1,3)--(1,2)--(1,1)--(1,0);
\draw (0,1)--(1,1)--(2,1);
\draw (1,0)--(2,0);
\draw [-latex](-1,2)--(-1,2.7);
\draw [-latex](0,2)--(0,2.7);
\draw [-latex](1,2)--(1,2.7);
\draw [-latex](0,1)--(0,1.7);
\draw [-latex](1,2)--(1,1.3);
\draw [-latex](1,1)--(1,0.3);
\filldraw[black] (-1,2) circle (2pt);
\filldraw[black] (-1,3) circle (2pt);
\filldraw[black] (0,2) circle (2pt);
\filldraw[black] (0,3) circle (2pt);
\filldraw[black] (1,2) circle (2pt);
\filldraw[black] (0,1) circle (2pt);
\filldraw[black] (1,1) circle (2pt);
\filldraw[black] (1,3) circle (2pt);
\filldraw[black] (2,2) circle (2pt);
\filldraw[black] (2,1) circle (2pt);
\filldraw[black] (2,0) circle (2pt);
\filldraw[black] (1,0) circle (2pt);
\draw [-latex](2,2)--(1.3,2);
\draw [-latex](2,1)--(1.3,1);
\draw [-latex](2,0)--(1.3,0);
\draw [-latex](-1,2)--(-0.3,2);
\draw [-latex](0,2)--(0.7,2);
\draw [-latex](1,1)--(0.3,1);
\end{tikzpicture}{\raisebox{9mm}{, }}
\begin{tikzpicture}[scale=0.50]
\draw (-1,2) -- (-1,3);
\draw (-1,2)-- (0,2)--(1,2)--(2,2);
\draw (0,3)--(0,2)--(0,1);
\draw (1,3)--(1,2)--(1,1)--(1,0);
\draw (0,1)--(1,1)--(2,1);
\draw (1,0)--(2,0);
\draw [-latex](-1,2)--(-1,2.7);
\draw [-latex](0,2)--(0,2.7);
\draw [-latex](1,2)--(1,2.7);
\draw [-latex](0,2)--(0,1.3);
\draw [-latex](1,1)--(1,1.7);
\draw [-latex](1,1)--(1,0.3);
\filldraw[black] (-1,2) circle (2pt);
\filldraw[black] (-1,3) circle (2pt);
\filldraw[black] (0,2) circle (2pt);
\filldraw[black] (0,3) circle (2pt);
\filldraw[black] (1,2) circle (2pt);
\filldraw[black] (0,1) circle (2pt);
\filldraw[black] (1,1) circle (2pt);
\filldraw[black] (1,3) circle (2pt);
\filldraw[black] (2,2) circle (2pt);
\filldraw[black] (2,1) circle (2pt);
\filldraw[black] (2,0) circle (2pt);
\filldraw[black] (1,0) circle (2pt);
\draw [-latex](2,2)--(1.3,2);
\draw [-latex](2,1)--(1.3,1);
\draw [-latex](2,0)--(1.3,0);
\draw [-latex](-1,2)--(-0.3,2);
\draw [-latex](1,2)--(0.3,2);
\draw [-latex](0,1)--(0.7,1);
\end{tikzpicture}
}\right\}.
\]

As mentioned before, the configurations in $\mathrm{6V}(n)$ are in bijective correspondence with the elements in DSASM$(n)$. The bijection that maps a DSASM $A$ of order $n$ to a configuration in $\mathrm{6V}(n)$ is given by 

\begin{equation}
\label{bij}
A_{i,j}=A_{j,i}=
\begin{cases}
-1 & C_{ij}
=\Bulkvi
\text{ or }\Lftdown,
\\
1 &
C_{ij} = \Bulkv \text{ or }
\Lftup,\\
0 & C_{ij} = \Bulki, 
\Bulkii, 
\Bulkiii, \Bulkiv, \Lftout
\text{ or } \Lftin,
\end{cases}
\end{equation}
for $1\le i\le j\le n$, where $C_{ij}$ is the 
local configuration at $(i,j)$.
The bijection in \eqref{bij} is a restriction of the bijection from the set of ASMs to the set of six-vertex model configurations
first observed by Robbins and Rumsey~\cite{robbins1986determinants},
and later revisited by by Elkies, Kuperberg, Larsen and Propp~\cite{elkies1992alternating} in the standard form now widely used. 

Define the {weight of each top and right boundary vertex} is $1$. For the {weight of a bulk vertex $(i,j)$ $($resp. left boundary vertex}$)$, denoted $W(C_{ij},u_iu_j)$ (resp. $W(C_{ii},u_i)$), we refer the reader to \cref{tale}, where $C_{ij}$ (resp. $C_{ii}$) is the local configuration at $(i,j)$ (resp. $(i,i)$), $u_1,\dots,u_n$ are indeterminates and 
$\alpha$, $\beta$, $\gamma$, $\delta$ and $q$ are arbitrary constants (i.e., independent of $u_i$, $1 \leq i \leq n$), and $\phi(u_i)$ is an
arbitrary function of~$u_i$. Different special cases of these weights have been previously employed in the enumeration of certain classes of ASMs, as seen in the works of Kuperberg~\cite{kuperberg2002symmetry}, Behrend, Fischer, and Konvalinka~\cite{behrend2017diagonally}, and Ayyer, Behrend, and Fischer~\cite{ayyer2020extreme}.

\begin{table}[h]
$\begin{array}{|@{\;\;\;\;\;\;}l|@{\;\;\;\;}l|}
\hline\rule{0ex}{3.7ex}
\,\,\,\,\,\,\,\,\text{Bulk weights}
&\quad \,\,\text{Left boundary weights}\\[2.0mm]
\hline\rule{0ex}{3.7ex}
W\left(\Bulkv,u_i u_j\right)=1
&
W\left(\,\Lftup,u_i\right)=
\left(\alpha\,q\,u_i
+\beta\,\q\,\bar{u}_i\right)\,\phi(u_i)\\[3.1mm]
W\left(\Bulkvi,u_i u_j\right)=1 & 
W\left(\,\Lftdown,u_i\right)=
\left(\alpha\,\q\,\bar{u}_i+\beta\,q\,u_i\right)\,\phi(u_i)\\[2.8mm]
W\left(\Bulki,u_i u_j\right)=
\sighhh(q^2u_i u_j)& 
W\left(\,\Lftin,u_i\right)=
\delta\,\sigma(q^2u_i^2)\,\phi(u_i) \\[4.6mm]
W\left(\Bulkii,u_i u_j\right)
=\sighhh(q^2u_i u_j)&
W\left(\,
\Lftout,
u_i\right)
=\gamma \,\sigma(q^2u_i^2)\,\phi(u_i)  \\[4.6mm]
W\left(\Bulkiii,u_i u_j\right)
=\sighhh(q^2\bar{u}_i \bar{u}_j)&\\[4.6mm]
W\left(\Bulkiv,u_i u_j\right)
=\sighhh(q^2\bar{u}_i \bar{u}_j)&\\[5mm]\hline
\end{array}$\\[2.4mm]
\caption{Bulk and left boundary weights.}
\label{tale}
\end{table}

Define the {DSASM partition function} $Z_n(u_1,\dots,u_n)$ for the DSASMs of order $n$ to be the 
sum of weights of all configurations $C$ in $\mathrm{6V}(n)$, where the {weight of a configuration} $C$ is defined to be the product of its vertex weights over all vertices of $T_n$. Explicitly,

\begin{equation}
\label{def-Z}
Z_n(u_1,\dots,u_n)=\sum_{C \, \in \, \mathrm{6V}(n)} \left( \prod_{i=1}^n W(C_{ii},u_i) \prod_{1 \leq i<j \leq n} W(C_{ij},u_iu_j) \right).
\end{equation}

For examples,
\begin{align*}
Z_1(u_1) &=(\alpha qu_1+\beta \bar{q} \bar{u}_1)\,\phi(u_1),\\
Z_2(u_1,u_2)&=\left((\alpha qu_1+\beta\q\bar{u}_1)\,
(\alpha qu_2+\beta\q\bar{u}_2)\,
\sighhh(q^2\bar{u}_1\bar{u}_2)
+\gamma\,\sigma(q^2u_1^2)\,\delta\,\sigma(q^2u_2^2)
\right)\phi(u_1)\,\phi(u_2).
\end{align*}

The following result gives a Pfaffian expression for the DSASM partition function. 

\begin{thm}[{\cite[Theorem 11]{behrend2023diagonally}}]
\label{partfunc}
The DSASM partition function is given by 
    \begin{equation*}
        Z_n(u_1,\dots,u_n)=\prod_{1 \leq i<j \leq n} 
        \frac{\sighhh(q^2 u_i u_j) \sighhh(q^2 \bar{u}_i \bar{u}_j)}{\sighhh(u_i \u_j)}
        \Pf_{\chi_{\mathrm{even}}(n) \leq i < j \leq n} 
        \left(
        \begin{cases}
            Z_{1}(u_j) & i=0,
            \\
        \frac{\sighhh(u_i \bar{u}_j) Z_2(u_i,u_j)}{\sighhh(q^2 u_i u_j) \sighhh(q^2 \bar{u}_i \bar{u}_j)}    & i \geq 1.
        \end{cases}
        \right)
    \end{equation*} 
\end{thm}

Now we consider a specialized DSASM partition function and denoted $\widetilde{Z}_n(u_1,\ldots,u_n)$,

\begin{equation}
\label{SpecZ}
\widetilde{Z}_n(u_1,\ldots,u_n)=Z_n(u_1,\ldots,u_n)\big|_{\alpha=\beta=s,\,
\gamma=\delta=1/\sigma(q),\, \phi(u_i)=1/(qu_i+\q\bar{u}_i), i=1,\dots,n}.
\end{equation}

For examples, 
\begin{equation}
\label{12} 
\begin{split}
\widetilde{Z}_1(u_1)=& \,s,\\
    \widetilde{Z}_2(u_1,u_2) = &  \, \frac{s^2\,\sigma(q^2\u_1\u_2)}{\sigma(q^4)} 
+\frac{\sigma(q^2u_1^2)\,\sigma(q^2u_2^2)}
{\sigma(q)^2 (qu_1+\q \u_1) (qu_2+\q \u_2)}\\
= & \,\frac{s^2\,(q^2\u_1\u_2-\q^2u_1u_2)}{(q^4-\q^4)} 
+\frac{(qu_1-\q\u_1)\,(qu_2-\q\u_2)}{(q-\q)^2}.
    \end{split}
\end{equation}

The following relation holds between the DSASM generating function given in \eqref{gf} and the specialized DSASM partition function. 

\begin{lem}[{\cite[Lemma 12]{behrend2023diagonally}}] 
\label{ZX}
The specialized DSASM generating function and DSASM partition function are related by
    \begin{equation*}
        \begin{split}
            \widetilde{Z}_n(z,\underbrace{1,\dots,1}_{n-1}) &=
            \frac{\sigma(q^2)^{(n-1)(n-2)/2} \sigma(q^2 \z)^{n-1}}{\sigma(q^4)^{(n)(n-1)/2} \sigma(q^2 z)}
            \\
           & \times \left(
           \frac{(q+\q) \sigma(qz) \sigma(q^2 \z)}{\sigma(q^2 z)} 
           X_n\left(q^2+\q^2,s,\frac{\sigma(q^2 z)}{\sigma(q^2 \z)}\right)
           - s \sigma(z) X_{n-1}(q^2+\q^2,s,1)
           \right),
        \end{split}
    \end{equation*}
   where $z$ is arbitrary,
\end{lem}

With the substitution given in~\eqref{SpecZ}, the left boundary weights in \cref{tale} are
\begin{equation}
\label{LftW} 
W\left(\, \Lftup,u_i\right)=W\left(\, \Lftdown,u_i\right)=s,\quad 
W\left(\, \Lftin,u_i\right)= W\left(\, \Lftout,u_i\right)=\frac{\sigma(qu_i)}{\sigma(q)}.
\end{equation}

Suppose $C \in \mathrm{6V}(n)$ is the image of $A \in \text{DSASM}(n)$ under the bijection given in \eqref{bij}. Then 
\begin{equation*}
\begin{split}
R(A)&=\text{the number of local configurations \Bulkv\ and \Bulkvi\ in }C,\\
S(A)&=\text{the number of local configurations \Lftup\ and \Lftdown\ in }C,\\
T(A)& =\text{the column index of the unique local configuration \Bulkv\ in the first row of }C,
\end{split}
\end{equation*}
where recall the notion of $R(A)$, $S(A)$ and $T(A)$ from \eqref{defnRST}. Recall that for $A \in \text{OSASM}(n)$, $S(A)=0$ if $n$ is even and  $S(A)=1$ if $n$ is odd. Therefore, the image set of the restriction of the bijection in \eqref{bij} to OSASM$(n)$ is the set of all 
configurations  $C \in \mathrm{6V}(n)$ such that the number of local configurations \Lftup\ and \Lftdown\ in $C$ is $0$ if $n$ is even and $1$ if $n$ is odd. Thus, in view of \eqref{LftW},
define the {OSASM partition function}
for the OSASMs of order $n$ as
\[
\begin{cases}
 \widetilde{Z}_n(u_1,\dots,u_n)\big|_{s=0} & n \text{ is even},\\[0.2cm]
 \bigl(\widetilde{Z}_n(u_1,\dots,u_n)/s\bigr)\Big|_{s=0} & n \text{ is odd.}
\end{cases}
\]
If $q^2+\bar{q}^2=1$, then the OSASM partition function specialized at $(\underbrace{1,\dots,1}_{n})$ counts the number of OSASMs in OSASM$(n)$.

\section{A product formula for the number of odd-order OSASMs}
\label{sec:product}
In this section, we write a proof of \eqref{OSams-odd}. The proof uses an equality between the generating function for odd-order OSASMs with specialized indeterminates and a certain symplectic character when $q^2+\bar{q}^2=1$ (see \cref{odd-osasm}). The symplectic character is the irreducible characters of the classical groups of type $C$ and is a symmetric Laurent polynomial indexed by an integer partition. Recall that an {integer partition} $\lambda$ is a weakly decreasing sequence of nonnegative integers such that only finitely many terms of $\lambda$ are nonzero. The {symplectic character} indexed by an integer partition $\lambda=(\lambda_1,\lambda_2,\dots,\lambda_n)$ is given by 

\begin{equation*}
\sp_\lambda(u_1,\dots,u_n)=
\frac{\ds \det_{1\le i,j\le n}\Bigl(u_i^{\lambda_j+n+1-j}-\u_i^{\lambda_j+n+1-j}\Bigr)}
{\ds \det_{1\le i,j\le n}\Bigl(u_i^{n+1-j}-\u_i^{n+1-j}\Bigr)}.
\end{equation*}

See~\cite{FulHar91} for a background on symplectic characters. For $q^2+\bar{q}^2=1$, it follows from \cite[Theorem 2.5(2)]{okada2006enumeration}, as well as from \cite[Theorem 5]{razumov2004refined} that the following equality holds true. 
\begin{equation}
\label{Zsymp}
\sp_{(n,n,\dots,0,0)}(u_1^2,\dots,u_{2n+2}^2)= \frac{3^{n(n+1)} \sigma(q)^{2n+2} \widetilde{Z}_{2n+2}(u_1,\dots,u_{2n+2})}{\prod_{i=1}^{2n+2} \sigma(qu_i)}\Bigg|_{s=0,q^2+\bar{q}^2=1}.
\end{equation} 

We now prove a similar equality involving the partition function for odd-order OSASMs.

\begin{thm}[{\cite[Conjecture 17]{behrend2023diagonally}}] 
\label{odd-osasm}
The partition function for odd-order OSASMs satisfies the following:
    \begin{equation}
    \label{Zsymp-odd}
    \begin{split}
  \left(\widetilde{Z}_{2n+1}(u_1,\dots,u_{n},\u_1,\dots,\u_{n},1)/s \right)&\Big|_{s=0,q^2+\bar{q}^2=1}  \\
 =  
 3^{-n^2} \prod_{i=1}^n \frac{(u_i+\u_i)^2}{u_i^2+\u_i^2-1} &
\, \sp_{(n,n,\dots,0,0)}(u_1^2,\dots,u_{n}^2,\u_1^2,\dots,\u_{n}^2,1,-1).
 \end{split}
    \end{equation}
\end{thm}
\begin{proof}[{First proof of \cref{odd-osasm}}]
Substituting $u_{n+i}=\u_i$, $1 \leq i \leq n$, $u_{2n+1}=1$, $u_{2n+2}=\iota$ in \eqref{Zsymp}, we have that
\begin{equation}
\label{new-1}
\begin{split}
\sp_{(n,n,\dots,0,0)}&(u_1^2,\dots,u_{n}^2,\bar{u}_1^2,\dots,\bar{u}_{n}^2,1,-1)\\
= & \frac{3^{n(n+1)} \sigma(q)^{2n+2} \widetilde{Z}_{2n+2}(u_1,\dots,u_{n},\bar{u}_1,\dots,\bar{u}_{n},1,\iota)}{\sigma(q) \sigma(q \iota) \prod_{i=1}^{n} \sigma(qu_i) \sigma(q\bar{u}_i)}\Bigg|_{s=0,q^2+\bar{q}^2=1}.
\end{split}
\end{equation}
If $q^2+\bar{q}^2=1$, then $q-\bar{q}=\iota$, $q+\bar{q}=\sqrt{3}$ and 
\[
 \frac{\sigma(q)^{2n+2}}{\sigma(q) \sigma(q \iota) \prod_{i=1}^{n} \sigma(qu_i) \sigma(q\bar{u}_i)} = \frac{(q-\q)^{2n+1} }{\iota (q+\q) \prod_{i=1}^{n} (q^2+\bar{q}^2-u_i^2-\u_i^2)} = 
 \frac{1}{\sqrt{3} \prod_{i=1}^{n} (u_i^2+\u_i^2-1)}.
\]
Plugging this into \eqref{new-1}, we have that
\begin{equation*}
    \sp_{(n,n,\dots,0,0)}(u_1^2,\dots,u_{n}^2,\bar{u}_1^2,\dots,\bar{u}_{n}^2,1,-1) = \frac{3^{n(n+1)} \widetilde{Z}_{2n+2}(u_1,\dots,u_{n},\bar{u}_1,\dots,\bar{u}_{n},1,\iota)\big|_{s=0,q^2+\bar{q}^2=1}}{\sqrt{3} \prod_{i=1}^{n} (u_i^2+\u_i^2-1)}.
\end{equation*}
Therefore, considering \eqref{Zsymp-odd}, it is enough to show that
\begin{equation}
\label{1}
\frac{\widetilde{Z}_{2n+1}(u_1,\dots,u_{n},\u_1,\dots,\u_{n},1)/s \big|_{s=0,q^2+\bar{q}^2=1}}
{\widetilde{Z}_{2n+2}(u_1,\dots,u_{n},\bar{u}_1,\dots,\bar{u}_{n},1,\iota)
\big|_{s=0,q^2+\bar{q}^2=1}}
=3^{(2n-1)/2} \, \prod_{i=1}^n \frac{(u_i+\u_i)^2}{(u_i^2+\u_i^2-1)^2}.
\end{equation}
By \cref{partfunc} and \eqref{SpecZ}, we have that 
\begin{equation}
\label{new-2}
\begin{split}
\widetilde{Z}_{2n+1}(u_1,u_2,\dots,u_{2n+1})/s&\\
=\prod_{1 \leq i<j \leq 2n+1} &
        \frac{\sighhh(q^2 u_i u_j) \sighhh(q^2 \u_i \u_j)}{\sighhh(u_i \u_j)} 
        \Pf_{0 \leq i < j \leq 2n+1} 
        \left(
        \begin{cases}
           1 & i=0,
            \\
        \frac{\sighhh(u_i \u_j) \widetilde{Z}_2(u_i,u_j)}{\sighhh(q^2 u_i u_j) \sighhh(q^2 \u_i \u_j)}    & i \geq 1
        \end{cases}
        \right)
        \end{split}
\end{equation}
and 
\begin{equation}
\label{new-3}
\begin{split}
\widetilde{Z}_{2n+2}(u_1,u_2,\dots,u_{2n+2})&\\
=
\prod_{1 \leq i<j \leq 2n+2} &
        \frac{\sighhh(q^2 u_i u_j) \sighhh(q^2 \u_i \u_j)}{\sighhh(u_i \u_j)}
        \Pf_{1 \leq i < j \leq 2n+2} 
        \left( \frac{\sighhh(u_i \u_j) \widetilde{Z}_2(u_i,u_j)}{\sighhh(q^2 u_i u_j) \sighhh(q^2 \u_i \u_j)}  \right). 
        \end{split}
\end{equation}
Using \eqref{12}, we note that
\[
\widetilde{Z}_2(u_1,u_2)\big|_{s=0,q^2+\bar{q}^2=1}=-(qu_1-\q\u_1)\,(qu_2-\q\u_2).
\]
Substituting $u_{n+i}=\u_i$, $1 \leq i \leq n$, $u_{2n+1}=1$, $u_{2n+2}=\iota$ in \eqref{new-2} and \eqref{new-3} and then taking ratios, we have that
\begin{equation}
\label{new-4}
\begin{split}
\frac{\widetilde{Z}_{2n+1}(u_1,\dots,u_{n},\u_1,\dots,\u_{n},1)/s \big|_{s=0,q^2+\bar{q}^2=1}}
{\widetilde{Z}_{2n+2}(u_1,\dots,u_{n},\bar{u}_1,\dots,\bar{u}_{n},1,\iota)
\big|_{s=0,q^2+\bar{q}^2=1}} &\\ =
 \prod_{i=1}^{2n+1} &\frac{\sighhh(-\iota u_i)}{\sighhh(q^2 u_i \iota) \sighhh(-q^2 \u_i \iota)}\Bigg|_{\substack{q^2+\bar{q}^2=1\\
u_{n+i}=\u_i, 1 \leq i \leq n\\
u_{2n+1}=1}} \times \frac{\Psi(1)}{\Psi(2)},
\end{split}
\end{equation}
where 
\[
\Psi(1)=\Pf_{0 \leq i < j \leq 2n+1} 
        \left(
        \begin{cases}
           1 & i=0,
            \\
        -\frac{\sighhh(u_i \u_j)(qu_i-\q\u_i)\,(qu_j-\q\u_j)}{\sighhh(q^2 u_i u_j) \sighhh(q^2 \u_i \u_j)}    & i \geq 1
        \end{cases}
        \right)\Bigg|_{\substack{q^2+\bar{q}^2=1\\
u_{n+i}=\u_i, 1 \leq i \leq n\\
u_{2n+1}=1}}
\]
and 
\[
\Psi(2)=\Pf_{1 \leq i < j \leq 2n+2} 
        \left( \frac{-\sighhh(u_i \u_j)(qu_i-\q\u_i)\,(qu_j-\q\u_j)}{\sighhh(q^2 u_i u_j) \sighhh(q^2 \u_i \u_j)}  \right)\Bigg|_{\substack{q^2+\bar{q}^2=1\\
u_{n+i}=\u_i, 1 \leq i \leq n\\
u_{2n+1}=1, u_{2n+2}=\iota}}.
\]
Observe that 
\begin{equation*}
\begin{split}
\prod_{i=1}^{2n+1} \frac{\sighhh(-\iota u_i)}{\sighhh(q^2 u_i \iota) \sighhh(-q^2 \u_i \iota)}&\Bigg|_{\substack{q^2+\bar{q}^2=1\\
u_{n+i}=\u_i, 1 \leq i \leq n\\
u_{2n+1}=1}}
\\ = &
\frac{\sighhh(- \iota)}{\sighhh(q^2 \iota) \sighhh(-q^2 \iota)} 
\prod_{i=1}^n \frac{\sighhh(-u_i \iota)} {\sighhh(q^2 u_i \iota) \sighhh(-q^2 \u_i \iota)}
\frac{\sighhh(-\u_i \iota)}{\sighhh(q^2 \u_i \iota) \sighhh(-q^2 u_i \iota)}\Bigg|_{q^2+\bar{q}^2=1}\\
= & {2 \sqrt{3} } \,\prod_{i=1}^n \frac{3(u_i+\u_i)^2}{(u_i^2+\u_i^2-1)^2}.
\end{split}
\end{equation*}
Plugging this into \eqref{new-4}, we have that
\[
\frac{\widetilde{Z}_{2n+1}(u_1,\dots,u_{n},\u_1,\dots,\u_{n},1)/s \big|_{s=0,q^2+\bar{q}^2=1}}
{\widetilde{Z}_{2n+2}(u_1,\dots,u_{n},\bar{u}_1,\dots,\bar{u}_{n},1,\iota)
\big|_{s=0,q^2+\bar{q}^2=1}}={2 \sqrt{3} } \,\prod_{i=1}^n \frac{\sqrt{3}(u_i+\u_i)^2}{(u_i^2+\u_i^2-1)^2} \times \frac{\Psi(1)}{\Psi(2)}.
\]
In view of \eqref{1}, it is enough to show $\Psi(2)= 6 \Psi(1)$ to complete the proof. 
Suppose $Q$ is a skew-symmetric matrix of order $(2n+2)$ such that for $1 \leq i < j \leq 2n+2$,
\begin{equation*}
\begin{split}
Q_{i,j} &= 
    \frac{-\sighhh(u_i \u_j)(qu_i-\q\u_i)\,(qu_j-\q\u_j)}{\sighhh(q^2 u_i u_j) \sighhh(q^2 \u_i \u_j)}  \Bigg|_{\substack{q^2+\bar{q}^2=1\\
u_{n+i}=\u_i, 1 \leq i \leq n\\
u_{2n+1}=1, u_{2n+2}=\iota}} \\
&= \frac{\sqrt{3} \iota (u_i \u_j - \u_i u_j) (qu_i-\q\u_i)\,(qu_j-\q\u_j)}
        {u_i^2 u_j^2+\u_i^2 \u_j^2+1}  \Bigg|_{\substack{q^2+\bar{q}^2=1\\
u_{n+i}=\u_i, 1 \leq i \leq n\\
u_{2n+1}=1, u_{2n+2}=\iota}}.  
\end{split}
\end{equation*}
Then 
\begin{equation}
\label{akmin}
\begin{split}
 \Psi(1)=\Pf_{0 \leq i<j \leq 2n+1} \left( 
\begin{cases}
    1 & i=0 \\
    Q_{i,j} & i \geq 1
\end{cases}
\right) = \Pf_{1 \leq i<j \leq 2n+2} \left( 
\begin{cases}
    Q_{i,j} & j \leq 2n+1\\
  1  & j=2n+2 
\end{cases}
\right)&, 
\\
\text{ and } \Psi&(2)=\Pf(Q).
\end{split}
\end{equation}
For $1 \leq i \leq 2n+1$, we note that 
\[
Q_{i,2n+2}=\frac{- 3\iota (u_i + \u_i) (qu_i-\q\u_i)}{u_i^2+\u_i^2-1}\Bigg|_{\substack{q^2+\bar{q}^2=1\\
u_{n+i}=\u_i, 1 \leq i \leq n\\
u_{2n+1}=1}}+6-6=\frac{3 \sqrt{3} (u_i - \u_i)}{(q \u_i-\q u_i)}\Bigg|_{\substack{q^2+\bar{q}^2=1\\
u_{n+i}=\u_i, 1 \leq i \leq n\\
u_{2n+1}=1}} + 6, 
\]
where we use $q+\bar{q}=\sqrt{3}$. 
Since a determinant is linear as a function of each of the rows and columns of the matrix, we have that
\begin{equation}
\label{tt}
\begin{split}
\det(Q)=& {\det  \left(
\begin{array}{c|c}
    (Q_{i,j})_{1 \leq i,j \leq 2n+1} & a_i \\[0.2cm]
   \hline \\[-0.3cm]
   -a_j  & 0
\end{array}
\right)}
+
\det \left(
\begin{array}{c|c}
    (Q_{i,j})_{1 \leq i,j \leq 2n+1} & a_i \\[0.2cm]
   \hline \\[-0.3cm]
   -6  & 0
\end{array}
\right) \\
& \quad +  \det \left(
\begin{array}{c|c}
    (Q_{i,j})_{1 \leq i,j \leq 2n+1} & 6 \\[0.2cm]
   \hline \\[-0.3cm]
   -a_j  & 0
\end{array}
\right) + \det \left(
\begin{array}{c|c}
    (Q_{i,j})_{1 \leq i,j \leq 2n+1} & 6 \\[0.2cm]
   \hline \\[-0.3cm]
   -6  & 0
\end{array}
\right),
\end{split}
\end{equation}
where 
\[
a_i=\frac{3 \sqrt{3} (u_i - \u_i)}{(q \u_i-\q u_i)}\Bigg|_{\substack{q^2+\bar{q}^2=1\\
u_{n+i}=\u_i, 1 \leq i \leq n\\
u_{2n+1}=1}}.
\]
For convenience, denote the four determinants from the right side of \eqref{tt} by $\Phi_1, \Phi_2, \Phi_3$ and $\Phi_4$ following in order. Observe that, $\Phi_2=\Phi_3$ and $\Phi_4=6 \Psi(1)$. 
Consider the first determinant $\Phi_1$ from the right side of \eqref{tt}.
\[
\Phi_1=\det 
\left(
\begin{array}{c|c|c|c}
(Q_{i,j})_{1 \leq i,j \leq n} 
& (R_{i,j})_{1 \leq i,j \leq n} 
& \frac{-(u_i - \u_i) (qu_i-\q\u_i)}
        {u_i^2+\u_i^2 +1}
        & \frac{ (u_i - \u_i)}
        {(q \u_i-\q u_i)} 
\\&&\\\hline&&\\
(S_{i,j})_{1 \leq i,j \leq n} 
        & 
(P_{i,j})_{1 \leq i,j \leq n} 
& \frac{(u_i - \u_i) (q \u_i-\q u_i)}
        {u_i^2+\u_i^2 +1} &
        \frac{ -(u_i - \u_i)}
        {(q u_i-\q \u_i)} 
\\&&\\\hline&&\\
\frac{(u_j - \u_j) (qu_j-\q\u_j)}
        {u_j^2+\u_j^2 +1} 
        & \frac{-(u_j - \u_j) (q \u_j-\q u_j)}
        {u_j^2+\u_j^2 +1} 
        & 0 & 
        0 \\&&\\\hline&&\\
    \frac{ (u_j - \u_j)}
        {(q \u_j-\q u_j)} &
        \frac{ -(u_j - \u_j)}
        {(q u_j-\q \u_j)} & 0 & 0
\end{array} 
\right), 
\]
where 
\[
Q_{i,j}= \frac{ \iota (u_i \u_j - \u_i u_j) (qu_i-\q\u_i)\,(qu_j-\q\u_j)}
        {u_i^2 u_j^2+\u_i^2 \u_j^2+1},
 \] 
\[
P_{i,j}= \frac{-  \iota (u_i \u_j - \u_i u_j) (q \u_i-\q u_i)\,(q \u_j-\q u_j)}
        {u_i^2 u_j^2+\u_i^2 \u_j^2+1},
\]
\[
R_{i,j}=\frac{ \iota (u_i u_j - \u_i \u_j) (qu_i-\q\u_i)\,(q \u_j-\q u_j)}{u_i^2 \u_j^2+\u_i^2 u_j^2+1}
\]
and
\[
S_{i,j}=\frac{ -\iota (u_i u_j - \u_i \u_j) (qu_j-\q\u_j)\,(q \u_i-\q u_i)}{u_i^2 \u_j^2+\u_i^2 u_j^2+1}.
\]
Applying the elementary row operations $R_{n+i} \rightarrow R_{n+i}+\left(\frac{q \bar{u}_i-\bar{q} u_i}{q u_i-\bar{q} \bar{u}_i}\right) R_i$ for $i=1,2,\dots,n$, we have that

\[
\Phi_1=\det 
\left(
\begin{array}{c|c|c|c}
(Q_{i,j})_{1 \leq i,j \leq n} 
& (R_{i,j})_{1 \leq i,j \leq n}
& \frac{-(u_i - \u_i) (qu_i-\q\u_i)}
        {u_i^2+\u_i^2 +1}
        & \frac{ (u_i - \u_i)}
        {(q \u_i-\q u_i)} 
\\&&\\\hline&&\\
(M_{i,j})_{1 \leq i,j \leq n} 
        & 
(N_{i,j})_{1 \leq i,j \leq n}
& 0 & 0
\\&&\\\hline&&\\
\frac{(u_j - \u_j)(qu_j-\q\u_j)}
        {u_j^2+\u_j^2 +1} 
        & \frac{-(u_j - \u_j) (q \u_j-\q u_j)}
        {u_j^2+\u_j^2 +1} 
        & 0 & 
        0 \\&&\\\hline&&\\
    \frac{ (u_j - \u_j)}
        {(q \u_j-\q u_j)} &
        \frac{ -(u_j - \u_j)}
        {(q u_j-\q \u_j)} & 0 & 0
\end{array} 
\right),
\]
where 
\[
M_{i,j}=(q \u_i-\q u_i)\,(qu_j-\q\u_j) \left(\frac{\iota (u_i \u_j - \u_i u_j ) }
        {u_i^2 u_j^2+\u_i^2 \u_j^2+1} -
        \frac{\iota (u_i u_j - \u_i \u_j)}{u_i^2 \u_j^2+\u_i^2 u_j^2+1}\right),
\]
and 
\[
N_{i,j}=(q \u_i-\q u_i)\,(q \u_j-\q u_j) \left(\frac{\iota (u_i u_j - \u_i \u_j) }
        {u_i^2 \u_j^2+\u_i^2 u_j^2+1}-
        \frac{\iota (u_i \u_j - \u_i u_j) }
        {u_i^2 u_j^2+\u_i^2 \u_j^2+1}\right).
\]
Applying the elementary column operations $C_{j+n} \rightarrow C_{j+n} + \left(\frac{q \bar{u}_j-\bar{q} u_j}{q u_j-\bar{q} \bar{u}_j}\right) C_j$ for $j=1,2,\dots,n$, we have that
\begin{equation}
\label{cuta}
\Phi_1=\det 
\left(
\begin{array}{c|c|c|c}
(Q_{i,j})_{1 \leq i,j \leq n} 
& (T_{i,j})_{1 \leq i,j \leq n} 
& \frac{-(u_i - \u_i) (qu_i-\q\u_i)}
        {u_i^2+\u_i^2 +1}
        & \frac{ (u_i - \u_i)}
        {(q \u_i-\q u_i)} 
\\&&\\\hline&&\\
(M_{i,j})_{1 \leq i,j \leq n}  
        & 
        (0)_{1 \leq i,j \leq n} 
& 0 & 0
\\&&\\\hline&&\\
\frac{(u_j - \u_j)(qu_j-\q\u_j)}
        {u_j^2+\u_j^2 +1} 
        & 0 
        & 0 & 
        0 \\&&\\\hline&&\\
    \frac{ (u_j - \u_j)}
        {(q \u_j-\q u_j)} &
        0 & 0 & 0
\end{array} 
\right),
\end{equation}
where \[
T_{i,j}=R_{i,j} + \left(\frac{q \bar{u}_j-\bar{q} u_j}{q u_j-\bar{q} \bar{u}_j}\right) Q_{i,j}.
\]
Since the determinant in \eqref{cuta} is zero, $\Phi_1=0$. 
Appyling the same set of operations to $\Phi_2$, we have that 
\[
\Phi_2=\det 
\left(
\begin{array}{c|c|c|c}
(Q_{i,j})_{1 \leq i,j \leq n}
& (T_{i,j})_{1 \leq i,j \leq n}
& \frac{-(u_i - \u_i) (qu_i-\q\u_i)}
        {u_i^2+\u_i^2 +1}
        & \frac{ (u_i - \u_i)}
        {(q \u_i-\q u_i)} 
\\&&\\\hline&&\\
(M_{i,j})_{1 \leq i,j \leq n}
        & 
(0)_{1 \leq i,j \leq n}
& 0 & 0
\\&&\\\hline&&\\
\frac{(u_j - \u_j)(qu_j-\q\u_j)}
        {u_j^2+\u_j^2 +1} 
        & 0 
        & 0 & 
        0 \\&&\\\hline&&\\
     \frac{2}
        {\sqrt{3}} &
        \frac{2}
        {\sqrt{3}}+\frac{2(qu_j-\q\u_j)}
        {\sqrt{3}(q \u_j-\q u_j)} & \frac{2(qu_j-\q\u_j)}{\sqrt{3}} & 0
\end{array} 
\right)=0.
\]
Since $\Phi_1=0, \Phi_2=0$ and $\Phi_2=\Phi_3$, the first three determinants in \eqref{tt} are zero. So, 
\[
\det(Q)=36 \det \left(
\begin{array}{c|c}
    (Q_{i,j})_{1 \leq i,j \leq 2n+1} & 1 \\[0.2cm]
   \hline \\[-0.3cm]
   -1  & 0
\end{array}
\right). 
\]
Therefore, in view of \eqref{akmin}, 
$\Psi(2)=\Pf(Q)=6 \Psi(1)$. 
This completes the proof.
\end{proof}

Substituting $u_{n+1}=\bar{u}_1, \dots, u_{2n}=\bar{u}_n, u_{2n+1}=1, u_{2n+2}=\iota$ in \eqref{Zsymp}
and comparing it with \eqref{Zsymp-odd}, we get the following equality. 
\begin{cor} We have 
\label{corr} 
    \begin{equation}
    \begin{split}
   \left( \widetilde{Z}_{2n+1}(u_1,\dots,u_{n},\u_1,\dots,\u_{n},1)/s\right) \Big|_{s=0,q^2+\bar{q}^2=1} &\\
= 3^{(2n-1)/2} \, \prod_{i=1}^n \frac{(u_i+\u_i)^2}{(u_i^2+\u_i^2-1)^2} 
\times  \widetilde{Z}_{2n+2} (u_1,&\dots,u_{n},\bar{u}_1,\dots,\bar{u}_{n},1,\iota)
\big|_{s=0,q^2+\bar{q}^2=1}.
\end{split}
    \end{equation}
\end{cor}

\begin{thm} The number of OSASMs of order $(2n+1)$ is given by  
    \begin{equation}
    \label{OX}
    |\text{OSASM}(2n+1)|=2^{2n} X^{\mathrm{O}}_{2n+2}(1,-1),
    \end{equation}
\end{thm}
\begin{proof}
  We have,
    \begin{equation}
    \label{akk}
    \begin{split}
    |\text{OSASM}(2n+1)|=&
   \bigl( \widetilde{Z}_{2n+1}(\underbrace{1,\dots,1}_{2n+1})/s\bigr)\Big|_{s=0,q^2+\q^2=1}
    \\
    =&
    2^{2n} 3^{(2n-1)/2}
    \widetilde{Z}_{2n+2}(\underbrace{1,\dots,1}_{2n+1},\iota)\big|_{s=0,q^2+\q^2=1} 
    \end{split}
    \end{equation}
  where we use \cref{corr} to get the second equality. By \cref{ZX}, for any arbitrary $z$, we have that
    \begin{equation}
    \label{zz}
        \begin{split}
            \widetilde{Z}_{2n+2}(z,\underbrace{1,\dots,1}_{2n+1})\big|_{s=0,q^2+\q^2=1} =
            \frac{\sigma(q^2)^{(2n+1)(n)} \sigma(q^2 \z)^{2n+1}}{\sigma(q^4)^{(n+1)(2n+1)} \sigma(q^2 z)} & \frac{(q+\q) \sigma(qz) \sigma(q^2 \z)}{\sigma(q^2 z)}
            \\
            \times &
           X_{2n+2}\left(q^2+\q^2,0,\frac{\sigma(q^2 z)}{\sigma(q^2 \z)}
           \right)\\
            = \frac{\sigma(q^2 \z)^{2n+2} (q+\q) \sigma(qz) }
           {\sigma(q^2)^{2n+1} \sigma(q^2 z)^2} & X^{\mathrm{O}}_{2n+2}\left(1,\frac{\sigma(q^2 z)}{\sigma(q^2 \z)}\right),
        \end{split}
    \end{equation}
    where the second equality uses \eqref{XOdef} and the fact that $\sigma(q^4)=\sigma(q^2)$ when $q^2+\q^2=1$. Substituting $z=\iota$ in \eqref{zz}, we get that 
    \begin{equation}
    \label{2}
    \begin{split}
        \widetilde{Z}_{2n+2}(\iota,\underbrace{1,\dots,1}_{2n+1})\big|_{s=0,q^2+\q^2=1}
        &=
        \frac{\sigma(-q^2 \iota)^{2n+2} (q+\q) \sigma(q \iota) }
           {\sigma(q^2)^{2n+1} \sigma(q^2 \iota)^2} X^{\mathrm{O}}_{2n+2}\left(1,\frac{\sigma(q^2 \iota)}{\sigma(-q^2 \iota)}\right)\\
= \frac{1}{3^{(2n-1)/2}}
           X^{\mathrm{O}}_{2n+2}\left(1,-1\right),
           \end{split}
    \end{equation}  
  where we use that $\sigma(q^2 \iota)=\iota, \sigma(-q^2 \iota)=-\iota$, $\sigma(q^2)=\sqrt{3} \iota$ and $(q+\bar{q})^2=3$ when $q^2+\q^2=1$. 
    Plugging this into \eqref{akk} completes the proof.
\end{proof}
Substituting the value of
$X^{\mathrm{O}}_{2n+2}(1,-1)$ from \eqref{XO} in \eqref{OX}, we get the following corollary.
\begin{cor}
\label{no}
The number of OSASMs of order $(2n+1)$ is given by
    \[
    |\text{OSASM}(2n+1)|=\frac{2^{n-1}(3n+2)!}{(2n+1)!}
\prod_{i=1}^n \frac{(6i-2)!}{(2n+2i+1)!}.
    \]
\end{cor}
Substituting $u_1=\dots=u_n=1$ and using \cref{no}, we get the following product formula for the specialized symplectic character.
\begin{cor} 
We have 
    \[
\sp_{(n,n,n-1,n-1,\dots,0)}(\underbrace{1,\dots,1}_{2n+1},-1)=\frac{3^{n^2}(3n+2)!}{2^{n+1} (2n+1)!}
\prod_{i=1}^n \frac{(6i-2)!}{(2n+2i+1)!}.
\]
\end{cor}

\section{A symmetry property for even-order OSASMs}
\label{sec:symm}

In this section, we prove the following symmetry property for even-order OSASMs conjectured by Behrend, Fischer and Koutschan~\cite[Conjecture 16]{behrend2023diagonally}. As a corollary, we get that 
\begin{equation*}
\begin{split}
    |\{M \in \text{OSASM}(2n): R(M)=&r_0, \, T(M)=t_0\}|\\
    & = |\{M \in \text{OSASM}(2n): R(M)=r_0,\, T(M)=2n+2-t_0\}|,
\end{split}
\end{equation*}
for any $r_0$ and $t_0$.
\begin{thm}[{\cite[Conjecture 16]{behrend2023diagonally}}] The generating function for even-order OSASMs satisfies the following:
    \[
    X^{\mathrm{O}}_{2n}(r,t)=t^{2n+2} X^{\mathrm{O}}_{2n}(r,\bar{t}).
    \]
\end{thm}
\begin{proof}
Observe that by \cref{ZX}, for any arbitrary $z$,
    \begin{equation*}
            \widetilde{Z}_{2n}(z,\underbrace{1,\dots,1}_{2n-1})\Big|_{s=0} =
            \frac{\sigma(q^2)^{(2n-1)(n-1)} \sigma(q^2 \z)^{2n} (q+\q) \sigma(qz)}{\sigma(q^4)^{(n)(2n-1)} \sigma(q^2 z)^2}
           X_{2n} \left(q^2+\q^2,0,\frac{\sigma(q^2 z)}{\sigma(q^2 \z)}\right).
    \end{equation*}
Suppose $z^2=\frac{tq^4+1}{t+q^{4}}$. Then    
 \[
 \frac{\sigma(q^2 z)}{\sigma(q^2 \z)}= \frac{q^2 z-\q^2 \z}{q^2 \z - \q^2 z} = \frac{q^4 z^2 -1}{q^4  - z^2} = t 
 \]
and 
 \[
  \frac{\widetilde{Z}_{2n}(z,\underbrace{1,\dots,1}_{2n-1})\Big|_{s=0}}{\widetilde{Z}_{2n}(\z,\underbrace{1,\dots,1}_{2n-1})\Big|_{s=0}} 
           = \frac{1}{t^{2n+2}} \times \frac{\sigma(qz)}
            {\sigma(q \z)}
             \times 
           \frac{X_{2n} \left(q^2+\q^2,0,t\right)}{X_{2n} \left(q^2+\q^2,0,\t\right)}.
 \]
 Since $q$ is an arbitrary constant, it is sufficient to show 
 \begin{equation}
 \label{aab}
  \frac{\widetilde{Z}_{2n}(z,\underbrace{1,\dots,1}_{2n-1})\Big|_{s=0}}{\widetilde{Z}_{2n}(\z,\underbrace{1,\dots,1}_{2n-1})\Big|_{s=0}}  =
 \frac{\sigma(qz)}{\sigma(q \z)}
 \end{equation}
 to complete the proof. By \cref{partfunc}, we have that
 \begin{equation}
 \label{ak}
        \widetilde{Z}_{2n}(u_1,u_2,\dots,u_{2n})=\prod_{1 \leq i<j \leq 2n} 
        \frac{\sighhh(q^2 u_i u_j) \sighhh(q^2 \u_i \u_j)}{\sighhh(u_i \u_j)}
        \Pf_{1 \leq i < j \leq 2n} 
        \left(
        \frac{\sighhh(u_i \u_j) \widetilde{Z}_2(u_i,u_j)}{\sighhh(q^2 u_i u_j) \sighhh(q^2 \u_i \u_j)}
        \right)
    \end{equation} 
    and using \eqref{12}, we note that
    \[
     \widetilde{Z}_2(u_i,u_j)\big|_{s=0}= \frac{\sigma(q u_i) \sigma(q u_j)}{\sigma(q)^2}. 
    \]
Substituting the value of $\widetilde{Z}_2(u_i,u_j)\big|_{s=0}$ in \eqref{ak}, we get that
     \begin{equation*}
     \begin{split}
        \widetilde{Z}_{2n}(u_1,u_2,\dots,u_{2n})\big|_{s=0}=&\prod_{1 \leq i<j \leq 2n} 
        \frac{\sighhh(q^2 u_i u_j) \sighhh(q^2 \u_i \u_j)}{\sighhh(u_i \u_j)}
        \Pf_{1 \leq i < j \leq 2n} 
        \left(
        \frac{\sighhh(u_i \u_j) \sigma(q u_i) \sigma(q u_j)}{\sighhh(q^2 u_i u_j) \sighhh(q^2 \u_i \u_j) \sigma(q)^2}
        \right)
        \\
        =&\prod_{1 \leq i<j \leq 2n} 
        \frac{\sighhh(q^2 u_i u_j) \sighhh(q^2 \u_i \u_j)}{\sighhh(u_i \u_j)} \prod_{i=1}^{2n} 
        \frac{\sigma(q u_i)}{\sigma(q)}
        \Pf_{1 \leq i < j \leq 2n} 
        \left(
        \frac{\sighhh(u_i \u_j)}{\sighhh(q^2 u_i u_j) \sighhh(q^2 \u_i \u_j)}
        \right),
        \end{split}
\end{equation*} 
where we use \eqref{out} to get the second equality. Therefore,
     \begin{equation}
     \label{aam}
    \frac{\widetilde{Z}_{2n}(u_1,u_2,\dots,u_{2n})\Big|_{s=0}}{\widetilde{Z}_{2n}(\u_1,u_2,\dots,u_{2n})\Big|_{s=0}}
    =  \prod_{j=2}^{2n} \frac{A(\u_1,u_j)}{A(u_1,u_j)} \times \frac{\sigma(q u_1)}{\sigma(q \u_1)} \times \frac{\Pf_{1 \leq i < j \leq 2n} 
        \left(
        A(u_i,u_j)
        \right)}{
        \Pf_{1 \leq i < j \leq 2n} 
        \left(
        A(u_i,u_j)
        \right)\big|_{u_1=\u_1}
        },
    \end{equation}
 where \[A(u_i,u_j)=\frac{\sighhh(u_i \u_j)}{\sighhh(q^2 u_i u_j) \sighhh(q^2 \u_i \u_j)}.\]
Substituting $u_1=z$ and taking limits ${u_2, \dots, u_{2n} \to 1}$ in \eqref{aam}, we have that
    \begin{equation}
    \label{aac}
     \frac{\widetilde{Z}_{2n}(z,\underbrace{1,\dots,1}_{2n-1})\Big|_{s=0}}{\widetilde{Z}_{2n}(\z,\underbrace{1,\dots,1}_{2n-1})\Big|_{s=0}}
    = - \frac{\sigma(q z)}{\sigma(q \bar{z})} 
    \times \frac{\Pf_{1 \leq i < j \leq 2n} 
        \left(
        A(u_i,u_j)
        \right)\big|_{u_1=z,u_2,\dots,u_{2n} \to 1}}{
        \Pf_{1 \leq i < j \leq 2n} 
        \left(
        A(u_i,u_j)
        \right)\big|_{u_1=\bar{z},u_2,\dots,u_{2n} \to 1}},
    \end{equation}
where we use the fact that $A(z,1)=-A(\bar{z},1)$. Observe that 
\[
\Pf_{1 \leq i < j \leq 2n} 
        \left(
        A(u_i,u_j)
        \right)\Big|_{u_1=z, u_2=\dots=u_{2n}=1}=
        \Pf_{1 \leq i < j \leq 2n} 
        \left(
        A(u_i,u_j)
        \right)\Big|_{u_1=\bar{z}, u_2=\dots=u_{2n}=1} =0.
\]
 Now consider a skew-symmetric matrix $B=\left(B_{i,j}\right)_{1 \leq i,j \leq 2n}$ such that for $1 \leq i<j \leq 2n$, 
  \[
  B_{i,j}=\begin{cases}
    A(u_1,1) & i=1\\
  A({u_i,u_j})  & i \geq 2.
  \end{cases}
  \]
Here $\Pf(B) \neq 0$. Since  $A(z,1)=-A(\bar{z},1)$, $\left(\Pf(B)\right)\big|_{u_1=z}=-\left(\Pf(B)\right)\big|_{u_1=\bar{z}}$. Thus, 
  \[
  \frac{\Pf_{1 \leq i < j \leq 2n} 
        \left(
        A(u_i,u_j)
        \right)\big|_{u_1=z,u_2,\dots,u_{2n} \to 1}}{
        \Pf_{1 \leq i < j \leq 2n} 
        \left(
        A(u_i,u_j)
        \right)\big|_{u_1=\bar{z},u_2,\dots,u_{2n} \to 1}}=
        \left(\frac{\Pf(B)\big|_{u_1=z}}{\Pf(B)\big|_{u_1=\bar{z}}}\right)\Bigg|_{u_2,\dots,u_{2n} \to 1} = -1. 
  \]
Plugging this into \eqref{aac} completes the proof in view of \eqref{aab}.
\end{proof}

\subsection*{Acknowledgement} The author acknowledges support from the Austrian Science Fund (FWF) grants 
\href{https://dx.doi.org/10.55776/F1002}{10.55776/F1002}.
The author would also like to thank I. Fischer for all the insightful discussions.

\bibliography{Bibliography}
\bibliographystyle{alpha}
\end{document}